\documentclass[a4paper,UKenglish,cleveref, autoref, thm-restate]{lipics-v2021-arxiv}

\titlerunning{When you come at the king you best not miss} 

\title{When you come at the king you best not miss}

\pdfoutput=1

\author{Oded~Lachish}%
{Birkbeck College, University of London, United Kingdom}%
{o.lachish@bbk.ac.uk}%
{https://orcid.org/0000-0001-5406-8121}
{}
\author{Felix~Reidl}%
{Birkbeck College, University of London, United Kingdom}%
{f.reidl@bbk.ac.uk}%
{https://orcid.org/0000-0002-2354-3003}%
{}
\author{Chhaya Trehan}%
{London School of Economics and Political Science, United Kingdom}%
{c.trehan@lse.ac.uk}%
{https://orcid.org/0000-0002-3249-3212}
{}
\authorrunning{O. Lachish and F. Reidl and C. Trehan} 
\Copyright{Oded Lachish and Felix Reidl and Chhaya Trehan}



\category{} 

\relatedversion{} 

\acknowledgements{}

\nolinenumbers 

\hideLIPIcs

\EventEditors{John Q. Open and Joan R. Access}
\EventNoEds{2}
\EventLongTitle{42nd Conference on Very Important Topics (CVIT 2016)}
\EventShortTitle{CVIT 2016}
\EventAcronym{CVIT}
\EventYear{2016}
\EventDate{December 24--27, 2016}
\EventLocation{Little Whinging, United Kingdom}
\EventLogo{}
\SeriesVolume{42}
\ArticleNo{23}

\usepackage{xpunctuate}
\usepackage{xcolor}
\usepackage[all,british]{foreign} 
\redefnotforeign[ie]{i.e\xperiodafter} 
\redefnotforeign[eg]{e.g\xperiodafter}

\usepackage{marginnote}

\def\polylog{\operatorname{polylog}}
\def\dir#1{\vec{#1}}

\renewcommand{\deg}{d}

\newcounter{propcounter}

\definecolor{amaranth}{rgb}{0.9, 0.17, 0.31}

\renewcommand{\etal}{\xperiodafter{\emph{et~al}}}

\def\Erdos{Erd\H{o}s\xspace}

\begin{document}

\maketitle
\begin{abstract}
    A tournament is an orientation of a complete graph.
	We say that a vertex~$x$ in a tournament~$\dir T$ \emph{controls} another
	vertex~$y$ if there exists a directed path of length at most two from~$x$ to~$y$.
	A vertex is called a \emph{king} if it controls every vertex of the tournament.
	It is well known that every tournament has a king.
	We follow Shen, Sheng, and Wu~\cite{ShenSW2003} in investigating the \emph{query complexity} of finding a king, that is, the number of arcs in~$\dir T$ one has to know in order to surely identify at least one vertex as a king.

	The aforementioned authors showed that one always has to query at least
	$\Omega(n^{4/3})$ arcs and provided a strategy that queries at most~$O(n^{3/2})$. While this upper bound has not yet been improved for the original problem,
	Biswas \etal~\cite{BiswasRaman2017} proved that with~$O(n^{4/3})$ queries
	one can identify a \emph{semi}-king, meaning a vertex which controls at
	least half of all vertices.

	Our contribution is a novel strategy which improves upon the number of controlled
	vertices: using~$O(n^{4/3} \polylog n)$ queries, we can identify a $(\frac{1}{2}+\frac{2}{17})$-king.
	To achieve this goal we use a novel structural result for tournaments.
\end{abstract}

\section{Introduction and Related Work}\label{sec:Intro}
A \emph{tournament} is a directed graph in which there is exactly one directed edge between every pair of vertices.
Due to their usefulness in modelling many real world scenarios such as game tournaments, voting strategies and many more, tournaments are
a very well studies concept in structural as well as algorithmic graph theory.
The early monograph of Moon~\cite{moon2015topics} has been followed by extensive research on the topic.
For example, Dey~\cite{dey2017query} studied the identification of the `best subset of vertices'  in a tournament
motivated by the high cost of comparing a pair of drugs for a specific disease.
Goyal \etal.~\cite{goyal2020elusiveness} studied the identification of vertices with specific in- or out-degrees.

 In this work we investigate the \emph{query complexity} of finding a \emph{king} in a tournament graph, that is, a vertex from which we can reach every other vertex of the tournament via a directed path of length at most two.
 It is well known that every tournament has such a vertex.
 
The study of \emph{query complexity} problems in tournaments has the following general shape: Initially, we are only given the vertex set of the tournament while the directions of its arcs are hidden from us. 
For each pair of vertices $u$, $v$ we can, at unit cost, learn whether the arc $uv$ or $vu$ is in the tournament. Our goal is to use the fewest possible queries in order to reveal some combinatorial object in the tournament. 
The motivation for our paper is found in Shen, Sheng, and Wu's work~\cite{ShenSW2003} on the query complexity of identifying a king. They showed that
$\Omega(n^{4/3})$ queries are always necessary and provided an
algorithm which reveals a king using~$O(n^{3/2})$ queries. 
Ajtai \etal.~\cite{SortingAjtai2009} independently proved the same upper bound in the context of imprecise comparison.

One of the enticing aspects of this setting is its game-theoretic nature: we can alternatively
think of it as an adversarial game where one player, the \emph{seeker}, wants to identify
a combinatorial structure by querying arcs of the tournament while an adversary, the \emph{obscurer},
tries to delay the seeker for as long as possible by choosing
the orientation of queried arcs. 

When reading Shen, Sheng, and Wu~\cite{ShenSW2003}, one may be tempted to conjecture
that a better analysis of their obscurer-strategy for finding a king can lead to a better lower bound.
However, Biswas \etal.~\cite{BiswasRaman2017} showed that against this strategy, the seeker can find a king with $O(n^{4/3})$ queries. 
They also showed that there exists a seeker strategy with $O(n^{4/3})$ queries for identifying
 a \emph{semi-king}, that is, a vertex which controls at least half of all
 vertices. 
This result is optimal by Lemma 6, Biswas \etal~\cite{BiswasRaman2017}.
In fact, one needs to make $\Omega(t^{4/3})$ queries for identifying a vertex which controls at least $t \le n$ vertices against the obscurer-strategy of Shen, Sheng and Wu. (See Lemma 6 of Biswas \etal~\cite{BiswasRaman2017} and Ajtai \etal~\cite{SortingAjtai2009}  for more details.)
Therefore, if there exists an obscurer-strategy that proves a stronger than $\Omega(n^{4/3})$ lower bound for the king problem, then this strategy must rely on some
factors which distinguish the king problem from the semi-king problem. In our eyes, this means that such a lower bound is much more difficult to find than one might think at first.

Proceeding from the above, it is tempting to try to improve the upper bound by using a variation of the seeker-strategy from Shen, Sheng, and Wu~\cite{ShenSW2003} and we can interpret the Biswas \etal~\cite{BiswasRaman2017}'s seeker-strategy for finding a semi-king as such an attempt.
These strategies both rely on repeatedly selecting a set of vertices and then querying all the edges between them to find a maximum out-degree (MOD) vertex in this sub-tournament\footnote{%
The relationship between MOD vertices and kings is well-established:
Landau~\cite{Landau1953}, while studying the structure of animal societies, 
showed that every MOD vertex is a king, but non-MOD kings can exist as well.
}.
Balasubramanian, Raman and Srinivasaragavan~\cite{balasubramanian1997finding} showed that identifying an MOD vertex in a tournament of size $k$ requires
$\Omega(k^2)$ queries in the worst case,
which may explain the limits of the existing seeker strategies.

\smallskip%
\noindent%
\textbf{Our Result}

\noindent
 In this paper, we proceed along the line of research just described.
On the one hand, we show that with $\tilde O(n^{4/3})$ queries\footnote{The big-$\tilde O$ notation
	hides constants and polylogarithmic factors}, it is possible to identify
a $(\frac{1}{2}+\frac{2}{17})$-king, which indicates that improving upon the $\Omega(n^{4/3})$ lower bound is probably even harder than indicated by the semi-king results.
On the other hand, our technique does not
rely on finding MOD vertices of sub-tournaments which circumvents the
inherent high cost of this operation.

\noindent \textbf{Technical Overview}

\noindent
Our result is based on the combinatorial structure of tournaments, which may be of independent interest.
We believe that this paper provides a novel toolkit which could lead towards resolving
the query complexity of finding a king. Specifically, we design a seeker-strategy which consists of two main stages: 
\begin{enumerate}[(i)]
	\item The seeker queries the orientation of a set of edges defined by a so-called \emph{template-graph}. 
	These queries are \emph{non-adaptive} in the sense that the queries do not change as a result
	of the answers provided by the obscurer.
	\item The seeker analyses the answer to the queries of (i) in order to select queries that lead to the revelation of a $(\frac{1}{2}+\frac{2}{17})$-king.
\end{enumerate}

\noindent
The template-graph is an undirected graph over the tournament's vertices that has $\tilde O(n^{4/3})$ edges, with the property that every set of vertices of size around $n^{2/3}$ or more has edges to almost all the graph.
In Section~\ref{sec:templategraph}, we use the probabilistic method to prove that such a graph exists.
Given the template-graph, the seeker's queries in the first stage are simply given by its edges, \ie if there exists an edge $uv$ in the template-graph, then the seeker asks the obscurer about the orientation of the edge $uv$ in the tournament.
The sparsity of the template-graph ensures that the seeker does not make too many queries and the connectivity of every sufficiently large set ensures that 
we do not miss any relevant information.

The second stage of the seeker-strategy is built on showing that when the obscurer chooses how the edges of the template graph are oriented they have a trade-off.
The trade-off is either to reveal an \textit{ultra-set} or not.
We show that if the obscurer reveals an \textit{ultra-set}, then the seeker can reveals a $(\frac{1}{2}+\frac{2}{17})$-king with $\tilde O(n^{4/3})$ extra queries.
If the obscurer does not do this, then the seeker can use this to find a partition of the vertex set of the tournament into sets of size $O(n^{2/3})$ each (which we refer to as \emph{tiles}), so that the edges of the template-graph that are incident to the tiles satisfy a certain property.
We obtain this combinatorial object by showing that if such a partition does not exist, then a simple set of queries already reveals a $(\frac{1}{2}+\frac{2}{17})$-king.

The tiles are analysed by the construction of what we refer to as the \emph{free matrix} which contains a row for every tile and a column for every vertex of the tournament.
An entry of the matrix indexed by a given tile-vertex pair is~$1$ if every edge between the vertex and a vertex in the tile is directed towards the tile, otherwise the entry is~$0$. We then use this free matrix to guide the seeker-strategy.

Given that the first part of the seeker-strategy is non-adaptive and against any adversary, this approach also reveals a combinatorial property of tournaments: 
for any fixed tournament and template graph with the same set of vertices, knowing only the direction of the arcs of the tournament that correspond to the arcs of the template graph is sufficient for finding a set of vertices $S$ of size $O(n^{2/3})$ such that querying all edges inside of $S$ necessarily reveals a $(\frac{1}{2}+\frac{2}{17})$-king. 
We note that fraction $\frac{2}{17}$ is the result of balancing the various trade-offs in the seeker strategy.

The rest of the paper is organised as follows.
In Section~\ref{sec:Prelim}, we provide necessary definitions and prove some basic lemmas about tournaments that are used in the rest of the paper.
Section~\ref{sec:templategraph} is dedicated to the formal definition and the proof of existence of template graphs.
In Section~\ref{sec:query}, we describe our seeker-strategy and prove that it leads to the discovery of a $(\frac{1}{2} + \frac{2}{17})$-king.
In Section~\ref{sec:conclusion}, we give concluding remarks and open problems.

\section{Preliminaries}\label{sec:Prelim}

For simplicity, we assume that the vertices of any $n$-vertex graph are the
numbers $[n] := \{1,\ldots,n\}$. 

An orientation of a graph $G$ is a directed graph obtained from $G$ by replacing every one of its edges by a directed arc.

\marginnote{$d(v,X)$,$N(v)$}
Given an undirected graph $G$, a vertex~$v \in V(G)$, and a
vertex set $X \subseteq V(G)$ we define the \emph{relative degree} $\deg(v,X) := |N(v) \cap X|$, where $N(v)$ is the neighbourhood of $v$ in $G$.

\marginnote{$d^+(X)$,$N^+(v)$}
For a vertex $v$ in a directed graph~$\dir G$, a vertex $u$ in $\dir G$ is an out-neighbour of $v$, 
if the edge between $u$ and $v$ is oriented from $v$ to $u$.
For a directed graph~$\dir G$, we denote the out-neighbourhood of a vertex
$v \in \dir G$ by $N^+(v)$ and its out-degree by~$\deg^+(v)$. 
For a vertex set $X \subseteq V(\dir G)$, we let $d^+(X)$ 
be the number of arcs from a vertex in $X$ to a vertex not in $X$. 
\marginnote{$\deg^+(v,X)$}
Given additionally a vertex subset~$X \subseteq V(G)$, we define the \emph{relative out-degree} $\deg^+(v,X) := |N^+(v) \cap X|$.

\marginnote{$N^{++}[v]$}
The (closed) \emph{second-out-neighbourhood} $N^{++}[v]$ of $v$ 
 is the set of vertices $u \in V(\dir G)$ for which there
exists a directed path from~$v$ to~$u$ of length at most two.

\marginnote{control, direct --}
For simplicity we will adopt the following vocabulary for digraphs. We say
that a vertex~$x$ \emph{controls} a vertex~$y$ if $y \in N^{++}[x]$. We say that~$x$ \emph{directly controls} a vertex~$y$ if
$y \in N^+(x) \cup \{x\}$. We extend both of these terms to vertex sets~$U$,
for example, we will often write statements like `$x$ controls at least half of the vertices in~$U$'.

\marginnote{$\dir T$, $\dir T[S]$}
A tournament is a digraph~$\dir T$ obtained from a complete graph by replacing each edge with a directed arc. As done usually, we denote the subgraph induced by a vertex set~$S \subseteq V(\dir T)$,
with~$\dir T[S]$.
Note that an induced subgraph of a tournament is necessarily also a tournament. 
We will need the following basics facts about tournaments in the following.

\begin{lemma}\label{lemma:HighOutDegreeK}
	Let $\dir T$ be a tournament with $m$ vertices and $\alpha \in [0,1]$ such that $\alpha m$ is even.
	Then $\dir T$ has at least $(1-\alpha)m$ vertices of out-degree at least $\alpha m/2$.
\end{lemma}
\begin{proof}
	Let $S$ initially be the vertices of $\dir T$ and proceed according to the
	following process: find a vertex of out-degree at least $\alpha m/2$ 
	and remove it from $S$;
	and repeat until no such vertex exists in $S$.
	From here on we focus on set $S$ after the vertex removal process ended.
	
	Let $r = |S|$ be the size of the final set and consider the sub-tournament
	$\dir T[S]$. 
	We know that by averaging considerations, every tournament of size $r$ 
	has at least one vertex of out-degree at least $r/2 - 1/2$.
	We also know that $S$ does not contain any
	vertex of out-degree at least $\alpha m/2$. 
	Hence, we conclude that $r \le \alpha m$.
	
	Consequently, our process discovered $m - r \geq (1-\alpha) m $ vertices of
	out-degree at least $\alpha m/2$ in $\dir T$.
\end{proof}

\begin{lemma}\label{lemma:HighOutDegreeKK}
	Let $\dir G$ be an orientation of a complete bipartite graph $(V_0, V_1, E)$, where $|V_0| = |V_1| = m$, and 
	 $m$ is divisible by $4$ . 
	Then, there exists $i\in \{0,1\}$, such that $V_i$ has at least $m/2+1$ vertices $v$, where $d^+(v, V_{1-i}) \geq m/4$.
\end{lemma}
\begin{proof}
	Let $S$ initially contain all the vertices in  $V_0 \cup V_1$ and proceed according to the
	following process: We find a vertex of out-degree at least $m/4$ and remove
	it from $S$. Repeat until no such vertex exists and we are left with 
	$S' \subseteq S$.
	
	Every orientation of the complete bipartite graph $K_{t,t}$ must contain, by
	a simple averaging argument, a vertex of out-degree at least $t/2$.
	Therefore the induced subgraph $\dir G[S']$ must have at least one partite
	set of size strictly less than $m/2$ or we could continue the process. 
	Consequently, our process discovered at least $m/2+1$ vertices of out-degree
	at least $m/4$ in that partite set.
\end{proof}

\begin{lemma}\label{lem:TournamentBipartition}
	Let $\dir T$ be a tournament on $2m$ vertices, where $m$ is divisible by
	$4$. Let further sets $S_0,S_1$ be a partition of the vertices of $\dir T$ into sets
	of equal size. Then there exists a vertex $v$ such that both $d^+(v, S_0)\geq m/4$ and $d^+(v, S_1) \geq m/4$.
\end{lemma}
\begin{proof}
	By Lemma~\ref{lemma:HighOutDegreeK},  for both $i\in \{0,1\}$,
	 there exist $m/2$ vertices $v$ in $S_i$ such
	that $|N^+(v)\cap S_i|\geq m/4$.  By
	Lemma~\ref{lemma:HighOutDegreeKK} for \emph{one} of $i \in \{0,1\}$, there
	exist  $m/2+1$ vertices $v$ such that  $|N^+(v)\cap S_{1-i}|\geq m/4$. Then
	by the pigeonhole principle, there exists $i\in \{0,1\}$ and a vertex  $v \in S_i$, such
	that $|N^+(v)\cap S_i|\geq m/4$, for every $i\in\{0,1\}$, as claimed. 
\end{proof}

\section{Constructing the template-graph}\label{sec:templategraph}

\begin{definition}[$\kappa$-template-graph]\label{def:TemplateGraph}
	Let $\kappa \in (0,1)$ and $G$ be an undirected graph over the vertex set $[n]$.
	The graph $G$ is a $\kappa$\emph{-template-graph}, if for every pair of disjoint sets $H_1, H_2 \subseteq [n]$ both of size at
	least $\kappa n^{2/3}$, there exists at least one edge between them, that is,
	$|E(H_1,H_2)| \geq 1$. 	
\end{definition}

\noindent
For the remainder of this section, we fix $\kappa\in (0,1)$ and set $p = \frac{2\log n + 2}{\kappa n^{2/3}}$.
We next show that with strictly positive probability the \Erdos--Renyi random graph $G(n,p)$ is a $\kappa$\emph{-template-graph}, 
with ${O}(n^{4/3}\log n)$ edges, where the $O$ notation hides a dependence on $\kappa$.
By the probabilistic method, this implies that there actually exists such a graph.

All probabilities in the following are with respect to the probability space of
this random graph. 

\begin{lemma}\label{lem:ThereISAnEdge}
	Let $\kappa \in (0,1)$.
	With probability at least $3/4$, the graph $G(n,p)$, where $p$ is defined as above,
	is a $\kappa$-template-graph.
\end{lemma}
\begin{proof}
	Note that if we prove the statement of the lemma for sets of size \emph{exactly} (up
	to rounding errors) $\kappa n^{2/3}$, then the
	claim follows for all larger sets as well. 
	To prove this, we next show, with the help of the union bound, 
	that the probability that $G(n,p)$ has two disjoint subsets of vertices, each of size $\kappa n^{2/3}$, 
	with no edge between them is strictly less than $1/4$. 
	
	Let $H_1$ and $H_2$ be any pair of disjoint subsets of $[n]$ of size $\kappa n^{2/3}$, 
	then the total number
	of vertex pairs between them is
	$(\kappa n^{2/3})^2$. The probability that none of these pairs is an edge 
	in the template-graph $G$ is accordingly $(1-p)^{(\kappa n^{2/3})^2}$. 
	
	We apply the exponential bound
	$(1-p)^k \leq e^{-pk}$ for a $k$-round Bernoulli trial and obtain
	\begin{align*}
		(1-p)^{(\kappa n^{\frac{2}{3}})^2} 
		& \leq e^{-p (\kappa n^{2/3})^2}
		= e^{-(2\log n+2)\kappa n^{2/3}} \\
		&= n^{-2\kappa n^{2/3}} e^{- 2\kappa n^{2/3}}
		\leq \frac{1}{4} n^{-2\kappa n^{2/3}},
	\end{align*}
	where the last inequality holds when~$n$ is large enough so that
	$\kappa n^{\frac{2}{3}} \geq 1$ and hence~$e^{-2\kappa n^{2/3}} < \frac{1}{4}$.
	Since the total number of pairs of sets $H_1, H_2$ of size $\kappa n^{2/3}$ is bounded above by $n^{2\kappa n^{2/3}}$, the claim now follows from the union bound. 
\end{proof}

\begin{theorem}\label{thm:KappaTemplateExists}
	Let $\kappa \in (0,1)$, there exists a $\kappa$-template-graph $G$ with at most $O(n^{4/3}\log{n}/\kappa)$ edges.
\end{theorem}
\begin{proof}
	The expected number of edges of~$G(n,p)$ for our choice of~$p = (2\log n + 2)/(\kappa n^{2/3})$ is less than~$m := (2\log n + 2)n^{4/3} / \kappa$. Since every edge of the graph is selected independently, by the Chernoff bound the probability that the number of edges in $G(n,p)$ exceeds~$2m$ 
	is at most 
	\begin{align*}
		e^{-p{n \choose 2}/3} 
		&\leq e^{-\frac{(2\log n + 2)}{\kappa n^{2/3}} \frac{n(n-1)}{6} }\\
		&\leq e^{-(2\log n+2)n/6} \leq \frac{1}{4}
	\end{align*}
	where the last inequality holds for~$n \geq 4$.
	
	Together with Lemma~\ref{lem:ThereISAnEdge} this implies that with
	probability at least $1/2$, $G(n,p)$ is a $\kappa$-template-graph $G$
	with at most $O(n^{4/3}\log{n}/\kappa)$ edges. The claim follows by
	the probabilistic method.
\end{proof}

\section{The seeker strategy}\label{sec:query}

\noindent
Having proved the existence of a $\kappa$-template-graph, 
we next examine the properties of an arbitrary orientation of such a graph.
Given a $\kappa$-template-graph $G_{\kappa}$ on a vertex set $[n]$, 
we use $\dir G_\kappa$ (\emph{henceforth}) to 
refer to a directed graph obtained by replacing every edge 
of $G_\kappa$ by a directed arc. Note that we assume nothing about $\dir G_\kappa$ and
analyze as if its arcs were arbitrarily oriented by an adversary.

\begin{definition}[$\eta$-weak, $\eta$-strong, $\eta$-ultra]\label{def:Sets}
	\marginnote{$\eta$-weak, $\eta$-strong, $\eta$-ultra}
	For an oriented template-graph $\dir G_\kappa$ and any $\eta > 0$, a set $H\subseteq [n]$ 
	is \emph{$\eta$-weak} if $d^+(H) <
	(1/2 +\eta)n$, is \emph{$\eta$-strong} if $d^+(H) \geq (1/2 +\eta)n$. We
	call a set \emph{$\eta$-ultra} if every subset $H' \subset H$, of size at
	least $|H|/2$ is $\eta$-strong.
\end{definition}

\noindent
To understand why $\eta$-ultra sets are important, it is useful to think of the seeker as trying to force the obscurer to reveal enough information (in the form of query answers) so that the seeker can achieve their goal.
This is done by first querying the orientation of all the edges of a template graph and nothing else. 
The observation below implies that if the orientation of the edges of the template graphs reveals an $eta$-ultra set, of size $\tilde{O}(n^{2/3})$, then the seeker can achieve its goal with an additional $\tilde{O}(n^{4/3})$ queries.
Thus, the obscurer cannot reveal such an ultra-set. 
However, as we show further on, by doing this the obscurer reveals enough information for the seeker to achieve their goal.

\begin{observation}\label{obs:DeltaUltra}
	Let $H \subseteq V( \dir G_\kappa)$ be an $\eta$-ultra set. Then we can find a
	$(1/2+\eta)$-king using $\leq |H|^2$ additional queries.
\end{observation}
\begin{proof}
	Query all $\leq |H|^2$ edges inside $H$. Let $v \in H$
	be a vertex such that $\deg^+(v,H) \geq |H|/2$. Since $H$ is $\eta$-ultra,
	the set $H' := N^+(v) \cap H$ is $\eta$-strong, meaning $\deg^+(H') \geq (1/2+\eta)n$.
	Therefore $|N^{++}[v]| \geq (1/2+\eta)n$ and $v$ is a 
	$(1/2+\eta)$-king.
\end{proof}

\begin{definition}[Free set]
	\marginnote{Free set, $F(W)$}
	Let $W \subseteq V(\dir G_\kappa)$ be an $\eta$-weak set. Then the \emph{free set} of $W$
	is the vertex set $F(W) := V(\dir G_\kappa)\setminus (N^+(W) \cup W)$, that is,
	all vertices that lie neither in $W$ nor in $N^+(W)$.
\end{definition}

\begin{observation}\label{obs:WeakSetFree}
	Let $W$ be an $\eta$-weak set. Then $|F(W)| > (\frac{1}{2} - \eta)n - |W|.$
\end{observation}

\noindent
By the properties of template-graphs, namely that each pair of large enough sets must have an edge between them, and by the definition of free sets it follows that all the arcs of $\vec{G}_\kappa$ between a sufficiently large set~$W$ and its free set~$F(W)$ must point towards~$W$. Let us formalize this intuition:

\begin{definition}[$\alpha$-covers]
	For $\alpha \in [0,1]$ we say that a set $S$ \emph{$\alpha$-covers} a 
	set $W$ if $|N^+(S) \cap W| \geq \alpha |W|$.
\end{definition}

\begin{lemma}\label{lemma:TileControl}
	In the template graph, for every set $W \subset [n]$ of size $n^{2/3}$ and
	every subset $S \subseteq F(W)$ of size at least $\kappa n^{2/3}$,
	it holds that $S$ $(1-\kappa)$-covers $W$.
\end{lemma}
\begin{proof}
	Consider a set $S \subseteq F(W)$ of size $\kappa n^{2/3}$. Then, by Lemma~\ref{lem:ThereISAnEdge},  there is at least  one edge $s_1w_1$ between some $s_1 \in S$ and $w_1 \in W$. Remove $w_1$ from $W$ and apply the argument to the remainder. In this way, we construct a
	sequence $w_1,\ldots,w_t$ such that each $w_i$ has at least one neighbour in $S$. 
	
	The application of Lemma~\ref{lem:ThereISAnEdge} is possible until the remainder of $W$ has size less than $\kappa n^{2/3}$, hence the process works for at least $t = n^{2/3} - \kappa n^{2/3} = (1-\kappa) n^{2/3}$ steps.
	Now simply note that each edge $sw_i$ for $s \in S$ must be oriented from $s$ to $w_i$ since $S$ is a subset of $F(W)$. It follows that $|N^+(S) \cap W| \geq (1-\kappa)|W|$, as claimed.
\end{proof}

\noindent 
In the previous lemma lies the inherent usefulness of free sets. If, for some set~$W$
of size~$n^{2/3}$, we find a vertex~$v$ that has at least~$\kappa n^{2/3}$ out-neighbours
in the free set~$F(W)$, then~$v$ controls almost all of~$W$.
As observed above, $\eta$-weak sets have necessarily large
free sets which makes them `easy targets' for our strategy.

 We now show that in case  no
$\eta$-ultra set exists (in which case we already win as per Observation~\ref{obs:DeltaUltra}), we can instead partition most of the vertices of $V(\dir G_\kappa)$ into weak sets.

\begin{definition}
	An \emph{$\eta$-weak tiling} of $\dir G_\kappa$ is a vertex partition $W_1,\ldots,W_m,R$
	where $|W_i| = n^{2/3}$, $|R| < 2n^{2/3}$ and every set $W_i$ is $\eta$-weak.
	We call the sets $W_i$ the \emph{tiles} and $R$ the \emph{remainder}.
\end{definition}

\noindent
By definition, the number of tiles $m$ in an $\eta$-weak tiling is at least $n^{1/3} - 2$. 

\begin{lemma}\label{lem:UltraSetOrTiling}
	Fix $\eta > 0$. For large enough $n$, $\dir G_\kappa$ either contains an $\eta$-ultra
	set of size $2n^{2/3}$ or an $\eta$-weak tiling.
\end{lemma}
\begin{proof}
	We construct the tiling iteratively. Assume we have constructed $W_1,\ldots,W_j$
	so far. Let $R := V(\dir G_\kappa) \setminus \bigcup_{i \le j} W_i$ be all the vertices of $\dir G_\kappa$ which are not yet part of the tiling. If $|R| < 2n^{2/3}$ we are done, so assume otherwise. Let $H \subseteq R$ be an arbitrary vertex set of size $2n^{2/3}$. If $H$ is $\eta$-ultra, then by Observation~\ref{obs:DeltaUltra}, we are done. Otherwise there exists an $\eta$-weak set $W_{j+1} \subseteq H$, $|W_{j+1}| = |H|/2$. Add this set to the tiling and repeat the construction. At the end of this procedure, we will either find an $\eta$-ultra set or an $\eta$-weak tiling.
\end{proof}

\noindent
Our goal is now to find a vertex whose out-neighbourhood has large intersections
with many free sets. To organise this search, we define the following auxiliary structure:

\begin{definition}[Free matrix]\label{def:FreeMatrix}
	Let $W_1,\ldots, W_m, R$ be an $\eta$-weak tiling of $\dir G_\kappa$ and let $\mathcal W = \{ W_1, \ldots, W_m \}$. The \emph{free matrix} $M$ of the tiling $\mathcal W, R$ is a binary matrix with $m$ rows indexed by $\mathcal W$ and $n$ columns indexed by $[n]$. The entry at position $(W_i,v) \in \mathcal W \times V$ is $1$ if 
	$v \in F(W_i)$ and $0$ otherwise.
\end{definition}

\marginnote{weight, $\sum M$}
\noindent
We will use the following notation in the rest of this section. Given a free
matrix~$M$ of an $\eta$-weak tiling $\mathcal W, R$ let $M[\mathcal W', U]$
denote a sub-matrix of $M$ induced by a subset $\mathcal W' \subseteq
\mathcal W$ of the tile set and a subset $U \subseteq [n]$ of the vertex set.
For example, a column  of $M$ corresponding to a vertex $v \in [n]$ can be
written as $M[\mathcal W, \{v\}]$ in this notation. Analogously a row of $M$
corresponding to a tile $W_i \in \mathcal W$ can be written as $M[\{W_i\},
[n]]$. Given a sub-matrix $M'$ of the free matrix $M$, we call the number of
$1$'s in $M'$ the \emph{weight} of $M'$ and denote it by $\sum M'$.

The following is a direct consequence of the construction of the free matrix and Observation~\ref{obs:WeakSetFree}.

\begin{observation}\label{obs:RowSum}
	Every row of the free matrix~$M$ has a weight of at least $(\frac{1}{2} - \eta - n^{-1/3})n$.
\end{observation}



\begin{definition}[Good Sub-Matrix]\label{def:GoodSubmatrix}
	A sub-matrix $M[\mathcal W, U]$, for some $U \subset [n]$, is $\eta$-\emph{good} if, 
	each one of its rows has weight at least $(\frac{1}{2} - \eta - 2n^{-1/3}\log^{1/2}{n})|U|$.
\end{definition}

\noindent
We next show that a good sub-matrix with $2 n^{2/3}$ columns exists, by using the probabilistic method.
Specifically, we show that if we randomly pick $2 n^{2/3}$ columns from the
matrix $M[\mathcal W, [n]]$ then with strictly positive probability the
matrix that includes exactly these columns is good.

\begin{lemma}\label{lem:GoodSubmatrix}
	Let $\eta \in (0,\frac{1}{2})$. For large enough~$n$
	the free matrix $M$ has an $\eta$-good sub-matrix with $2n^{2/3}$ columns.
\end{lemma}
\begin{proof}
	Select $K \subset [n]$ of size $2n^{2/3}$ uniformly at random. Let $M' = M[\mathcal{W},K]$.

	We set $p = 1/2 - \eta - n^{-1/3}$ and $t = n^{-1/3}\log^{1/2}{n}$. By
	Observation~\ref{obs:RowSum}, every row of~$M$ has weight at
	least~$pn$. By the Hypergeometric tail bound the
	probability that a specific row of $M'$ has weight less than $
	(p - t)n$ is at most $e^{-2t^2 2n^{2/3}}\leq 1/n$, where the
	inequality follows from our choice of~$t$. Then by the union bound the
	probability that \emph{every} row of $M'$ has weight at least $(p-t)n$ is
	strictly positive. 

	Note that by our choice of~$p$ and~$t$,  we get that $(p-t)n$
	, for large enough~$n$, is at least as large as 
	$(1/2 - \eta - 2n^{-1/3}\log^{1/2}{n})n$, therefore a good sub-matrix 
	$M'$ exists with strictly positive probability. By the probabilistic
	methods the claim therefore holds.
\end{proof}

\noindent
Next we show that the only way that the adversary does not provide us with a
$(1/2+\delta)$-king once we have identified a $\delta$-good sub-matrix is if
the distribution of $1$'s in that matrix is very restricted. We will use this
additional structure to find a $(1/2+\delta)$-king in the sequel. For
simplicity, we first query all the edges between the vertices associated with the columns of the $\delta$-good sub-matrix, but note that this is not strictly necessary: we can instead inspect all potential partitions (with properties as stated in the lemma) and only if no such partition exists query said edges which then surely identifies a $(1/2+\delta)$-king. With this change the lemma is consistent with the structural claim from the introduction.

\begin{lemma}\label{lemma:MatrixDecomp}
	Let $M' = M[\mathcal W, V]$ be a $\delta$-good sub-matrix of $M$ with $2n^{2/3}$ columns.
	Let further $\kappa + \delta \leq 1/2$.
	If we query each one of the $O(n^{4/3})$ edges in $V$, then
	either we find a $(\frac{1}{2}+\delta)$-king, 
	or we find partitions $V_1 \uplus V_2 = V$ and
	$\mathcal W_1 \uplus \mathcal W_2 = \mathcal W$  with the following properties:
	\begin{itemize}
		\item $|V_1| = |V_2| = n^{2/3}$
		\item $|\mathcal W_1| \geq (\frac{1}{2} - \delta - \kappa)n^{1/3} -2 $
		and $|\mathcal W_2| < (\frac{1}{2} + \delta + \kappa)n^{1/3}$
		\item Every row in $M'[\mathcal W_1,V_1]$ has weight at most $\kappa n^{2/3}$
		\item Every row in $M'[\mathcal W_2,V_1]$ has weight at least $\kappa n^{2/3}$
	\end{itemize}
\end{lemma}
\begin{proof}
	We query all the edges in $V \times V$ and select a vertex $y \in V$ such that $d^+(y,V) \geq n^{2/3}$. 
	Let $V_1$ be an arbitrary subset of $N^+(y) \cap V$ of size
	$n^{2/3}$ and let $V_2 = V\setminus V_1$.
	Partition the rows of $M'$ into
	$\mathcal W_1 \cup \mathcal W_2$ so that $\mathcal W_1$ contains all rows with weight
	less than $\kappa n^{2/3}$ in the sub-matrix $M'[ \mathcal W,V_1]$. 
	
	We claim that if $|\mathcal W_1| < (\frac{1}{2}-\delta-\kappa) n^{1/3} -2$ then $y$ is a $(\frac{1}{2}+\delta)$-king.
	By construction, every row in $\mathcal W_2$ has weight at least $\kappa n^{2/3}$ in $M'[\mathcal W,V_1]$ and if the condition of the claim holds
	then $|\mathcal W_2| \geq (\frac{1}{2}+\delta+\kappa) n^{1/3}$. 
	By Lemma~\ref{lemma:TileControl}, the set $V_1$ must $(1-\kappa)$-cover every tile in $\mathcal W_2$. It follows that 
	\begin{align*}
		|N^{++}[y]| &\geq (1-\kappa)\big|\bigcup \mathcal W_2\big| \geq (1-\kappa) \left(\Big(\frac{1}{2}+\delta+\kappa\Big) n^{1/3}\right) n^{2/3} \\
		&= (1-\kappa) \Big(\frac{1}{2}+\delta+\kappa\Big) n
		= \Big(\frac{1}{2}+\delta+\frac{\kappa}{2}-\kappa\delta-\kappa^2\Big) n \\
		&= \left(\frac{1}{2}+\delta+\kappa\Big(\frac{1}{2}-\delta-\kappa\Big)\right) n \\
		&\geq \Big(\frac{1}{2}+\delta\Big) n 
	\end{align*}
	where the last inequality holds since $\delta+\kappa \leq 1/2$. 
\end{proof}

\noindent
Lemma~\ref{lemma:MatrixDecomp} implies the following about good sub-matrices.

\begin{lemma}\label{lemma:W1V2Dense}
	Let $M' = M[\mathcal W, V]$ be a $\delta$-good sub-matrix with $|V| = 2n^{2/3}$
	and $ \mathcal W_1 \uplus \mathcal W_2$, $ V_1 \uplus V_2$ be partitions
	as in Lemma~\ref{lemma:MatrixDecomp}.
	Then every row in $M'[\mathcal W_1,V_2]$ has weight at least $(1 - 2\delta - 2\kappa)n^{2/3}$.
\end{lemma}
\begin{proof}
	Since $M'$ is a $\delta$-good sub-matrix, by Definition~\ref{def:GoodSubmatrix}, 
	every row in $M'[\mathcal W_1,V]$ has weight at least
	$(1/2 - \delta - 2n^{-1/3}\log^{1/2}{n}) 2n^{2/3}$.
	By Lemma~\ref{lemma:MatrixDecomp}, every row in $M'[\mathcal W_1,V_1]$ has weight  at most $\kappa n^{2/3}$.
	Therefore, the weight of every row in  $M'[\mathcal W_1,V_2]$ is 
	\begin{align*}
		& \ge (1/2 - \delta - 2n^{-1/3}\log^{1/2}{n}) 2n^{2/3} -  \kappa n^{2/3} \\
		&= (1 - 2\delta -4 \frac{\log^{1/2}{n}}{n^{1/3}} -\kappa) n^{2/3} \\
		&\geq (1 - 2\delta - 2\kappa) n^{2/3}
	\end{align*}
	where we assume that $n$ is large enough so that $4 \frac{\log^{1/2}{n}}{n^{1/3}} \leq \kappa$.	
\end{proof}

\noindent
Our final technical lemma lets us, for a given set of rows of $M$, identify a set of columns with high enough weight when restricted to those rows.

\begin{lemma}\label{lemma:V3}
	Let $U\subset [n]$ be of size $2n^{2/3}$ and $\mathcal{W'} \subset \mathcal{W}$.
	Then there exists a set $V'\subset [n]\setminus U$, 
	of size $n^{2/3}$ such that every column in  $M[\mathcal{W'},V']$
	has weight at least $(1/2 -\delta - 3n^{-1/3})|\mathcal{W'}|$.
\end{lemma}
\begin{proof}
	Let $\bar U := [n] \setminus U$ and
	let $V'$ be an arbitrary subset of $n^{2/3}$ columns in $M
	[\mathcal W',\bar U]$ with the largest column-weight. Let
	$t$ be the smallest weight among these columns when restricted to $M
	[\mathcal W',V']$. We bound the weight of $M[\mathcal W',\bar U]$
	first from below and then from above, then we use
	these bounds to show that $t > (\frac{1}{2} - \delta - 3n^
	{-1/3}) \cdot |\mathcal{W'}|$, which implies that $V'$ is the claimed set.
	
	For the lower bound on the weight of $M[\mathcal W', \bar U]$, we use the simple fact that
	\begin{align}\label{eq:forlower}
		\sum M[\mathcal W', \bar U] = \sum M[\mathcal W', [n]] - \sum M[\mathcal W',U].
	\end{align}

	\noindent
	By Observation~\ref{obs:RowSum}, every row of the matrix $M$ has
	weight least $(\frac{1}{2} - \delta - n^{-1/3}) \cdot n$. It
	follows that $\sum M[\mathcal W', [n]]$ is at least $(\frac{1}
	{2} - \delta - n^{-1/3}) \cdot n \cdot |\mathcal{W}'|$. For the second
	term, we have the trivial bound $\sum M[\mathcal W',U] \leq |U| \cdot |\mathcal W'| = 2n^{2/3}\cdot |\mathcal{W}'|$. 
	Plugging these values into (\ref{eq:forlower}) we obtain
	\begin{equation}\label{eq:oneslower}
		\begin{aligned}
			\sum M[\mathcal W', \bar U] &\ge (\frac{1}{2} - \delta - n^{-1/3}) \cdot n \cdot |\mathcal{W'}| - 2n^{2/3}\cdot |\mathcal{W'}| \\
			&= (\frac{1}{2} - \delta - 3n^{-1/3}) \cdot n \cdot |\mathcal{W'}|.
		\end{aligned}
	\end{equation}
	
	\noindent
	For the upper bound on the total weight of $M[\mathcal W', \bar U]$ we use that
	\begin{align}\label{eq:forupper}
		\sum M[\mathcal W', \bar U] = \sum M[\mathcal W',V'] + \sum M[\mathcal W', \bar U \setminus V'].
	\end{align}
	We use the trivial bound $\sum M[\mathcal W',V'] \leq |V'| \cdot |\mathcal W'| = n^{2/3} \cdot |\mathcal W'|$ for the first term.
	By definition of the value~$t$, we have that every column in $M[\mathcal W',\bar U \setminus V']$ has weight at most~$t$.
	Accordingly, $\sum M[\mathcal W', \bar U \setminus V'] \leq t \cdot |\bar U \setminus V'|  = t \cdot (n - 3n^{2/3})$.
	Plugging in these values into (\ref{eq:forupper}) we obtain
	\begin{equation}\label{eq:onesupper}
		\begin{aligned}
			\sum M[\mathcal W', \bar U] &\le  n^{2/3} \cdot |\mathcal W'| + t \cdot (n- 3n^{2/3}).
		\end{aligned}
	\end{equation}
	
	\noindent
	Finally, (\ref{eq:oneslower}) and (\ref{eq:onesupper}) taken together give us that
	\[
		t \cdot (n - 3n^{2/3}) + n^{2/3} \cdot |\mathcal{W}'| \geq (\frac{1}{2} - \delta - 2n^{-1/3}) \cdot n \cdot |\mathcal{W'}|.
	\]
	Consequently, $t > (\frac{1}{2} - \delta - 3n^{-1/3}) \cdot |\mathcal{W'}|$ and we conclude that $V'$ has the claimed property.
\end{proof}

\noindent
We are finally ready to prove our seeker-strategy for finding a $1/2 + \delta$-king  using $\tilde O (n^{4/3})$ queries.
For readability, we will state our main result in terms of concrete and simple values for~$\kappa$ and
$\delta$, however, note that smaller values of~$\kappa$ allow $\delta$ to be slightly larger than the stated bound of $\frac{2}{17}$. 

\begin{theorem}\label{thm:Final}
	Fix $\delta = \frac{2}{17}$, let $\kappa = \frac{1}{4000}$. For 
	larger enough~$n$, there exists a seeker strategy for finding a $(1/2 + \delta)$-king using $\tilde O(n^{4/3})$ edge queries.
\end{theorem}
\begin{proof}
	We construct the template-graph $G_\kappa$ and query all $\tilde O(n^{4/3})$ of its edges to obtain $\dir G_\kappa$.

	By Lemma~\ref{lem:UltraSetOrTiling}, we either obtain a $\delta$-ultra set of size $2 \cdot n^{2/3}$ or a $\delta$-weak tiling of $\dir {G_\kappa}$. If we find the former, by Observation~\ref{obs:DeltaUltra} we can find a $(\frac{1}{2}+\delta)$-king using
	$O(n^{4/3})$ additional queries. Therefore assume that we obtained
	a $\delta$-weak tiling $\mathcal W,R$ of $\dir G_\kappa$.

	Let~$M$ be the free matrix of $\mathcal W,R$. By Lemma~\ref{lem:GoodSubmatrix}, $M$ has a $\delta$-good sub-matrix $M[\mathcal W, V]$ with $|V| = 2n^{2/3}$. We query all $O(n^{4/3})$ edges in
	$V \times V$ and by Lemma~\ref{lemma:MatrixDecomp} either identify
	a $(\frac{1}{2}+\delta)$-king, or obtain partitions 
	$V_1\uplus V_2 = V$, $\mathcal W_1 \uplus \mathcal W_2$ with properties as listed in Lemma~\ref{lemma:MatrixDecomp}. 
	Importantly, by Lemma~\ref{lemma:W1V2Dense}, every row in the sub-matrix $M[\mathcal W_1, V_2]$ has weight at least $(1-2\delta-2\kappa)n^{2/3}$.

	We now apply Lemma~\ref{lemma:V3} with $\mathcal W' = \mathcal W_2$
	and find a set of columns $V_3 \subseteq [n] \setminus V$ of size $n^{2/3}$ such that every column in $M[\mathcal W_2, V_3]$ has weight
	at least $(\frac{1}{2} - \delta -3n^{-1/3})|\mathcal W_2|$.
	We now query all edges in $V_3\times V_3$ and $V_2 \times V_3$, 
	since $|V_2| = |V_3| = n^{2/3}$ this amounts to $O(n^{4/3})$ additional queries.
	
	Since $\dir G[V_2 \cup V_3]$ is completely revealed, it is a tournament
	of size $2n^{2/3}$ and we apply Lemma~\ref{lem:TournamentBipartition}
	using the bipartition $(V_2, V_3)$ to find a vertex $v \in V_2 \cup V_3$ such that $\deg^+(v,V_2)$ and $\deg^+(v,V_3)$ are
	both at least $n^{2/3}/4$. We claim that~$v$ is a $(\frac{1}{2}+\delta)$-king.
	Let in the following $V'_2 = N^+(v) \cap V_2$ and $V'_3 = N^+(v) \cap V_3$. We first prove the following two claims about these two sets:
	
	\begin{claim}
		Every row in $M[\mathcal W_1, V'_2]$ has weight at
		least $\kappa n^{2/3}$.
	\end{claim}
	\begin{proof}[Proof of the claim]
		According to Lemma~\ref{lemma:W1V2Dense}, every row in $M[\mathcal W_1,V_2]$ has weight at least $(1-2\delta-2\kappa) n^{2/3}$. Since $|V'_2| = |V_2|/4 = n^{2/3}/4$, we have that each row
		in $M[\mathcal W_1, V'_2]$ has weight at least
		\[
			(1-2\delta-2\kappa) n^{2/3} - \frac{3}{4} n^{2/3}
		\]
		which is larger than $\kappa n^{2/3}$ for $\delta \leq \frac{1}{8} - \frac{3 \kappa}{2}$ which holds true for our choices of~$\delta$
		and~$\kappa$.
	\end{proof}

	\begin{claim}
		At least $(\frac{1}{2} - \delta -4\kappa - 3n^{-1/3}) \frac{|\mathcal W_2|}{1-4\kappa}$ rows in $M[\mathcal W_2, V'_3]$ have weight at
		least $\kappa n^{2/3}$.
	\end{claim}
	\begin{proof}[Proof of the claim]
		Recall that by choice of $V_3$, every column in $M[\mathcal W_2, V_3]$ and therefore also $M[\mathcal W_2, V'_3]$ has weight at least $(\frac{1}{2} - \delta - 3n^{-1/3})|\mathcal W_2|$.
		Accordingly, 
		\begin{align}
			\begin{aligned}\label{eq:W2V3}
			\sum M[\mathcal W_2, V'_3] & \geq (\frac{1}{2} - \delta - 3n^{-1/3})|\mathcal W_2| \cdot |V'_3|  \\
			& \geq (\frac{1}{2} - \delta - 3n^{-1/3})|\mathcal W_2| \cdot \frac{1}{4} n^{2/3}. \\
			\end{aligned}
		\end{align} 
		Let~$t$ denote the number of rows in $M[\mathcal W_2, V'_3]$ with weight at least $\kappa n^{2/3}$. Our goal is to find a lower bound for~$t$. Since~$t$ is minimized if every row that
		has weigth at least~$\kappa n^{2/3}$ has in fact the maximum possible weight $|V'_3| = n^{2/3}/4$, we can lower-bound~$t$
		using 
		\begin{align*}
			\frac{t}{4} n^{2/3} + (|\mathcal W_2|-t) \kappa n^{2/3} \geq \sum M[\mathcal W_2, V'_3].
		\end{align*}
		Combining this inequality with~(\ref{eq:W2V3}), we obtain
	\begin{align*}
		& t\frac{n^{2/3}}{4}  + (|\mathcal W_2|-t) \kappa n^{2/3} \geq (\frac{1}{2} - \delta - 3n^{-1/3})|\mathcal W_2| \cdot \frac{1}{4} n^{2/3} \\
		\iff{}& t (1 -  4\kappa) \geq (\frac{1}{2} - \delta -4\kappa - 3n^{-1/3})|\mathcal W_2|  \\
		\iff{}& t \geq (\frac{1}{2} - \delta - 4\kappa - 3n^{-1/3}) \frac{|\mathcal W_2|}{1-4\kappa}\qedhere
	\end{align*}
	\end{proof}

	\noindent
	Now note that for every tile~$W \in \mathcal W$ for which the
	row~$M[\{W\}, V'_2 \cup V'_3]$ has weight at least~$\kappa n^{2/3}$
	we have that~$d^+(v,F(W)) \geq \kappa n^{2/3}$, therefore by Lemma~\ref{lemma:TileControl} the set $N^+(v) \cap F(W)$ $(1-\kappa)$-covers~$W$. 
	In other words, $v$ controls at least~$(1-\kappa) n^{2/3}$ vertices in~$W$.

	Our goal is now
	to lower-bound the total number of such tiles, hence
	let~$s$ denote the total number of rows in~$M[\mathcal W, V'_2 \cup V'_3]$ with weight
	at least~$\kappa n^{2/3}$. By the previous two observations and by plugging in the concrete values of $\delta = \frac{2}{17}$ and~$\kappa = \frac{1}{4000}$, we have that
	\begin{align*}
		s &\geq |\mathcal W_1| + (\frac{1}{2} - \delta - 4\kappa - 3n^{-1/3}) \frac{|\mathcal W_2|}{1-4\kappa} \\
		&= |\mathcal W_1| + (\frac{6483}{17000} - 3n^{-1/3}) |\mathcal W_2|\frac{1000}{999}
	\end{align*}
	Again we are aiming to prove a lower-bound, thus we assume that $\mathcal W_1$ is as small as possible. By Lemma~\ref{lemma:MatrixDecomp}, this means that 
	\begin{align*}
		|\mathcal W_1| &= (\frac{1}{2} - \delta - \kappa)n^{1/3} - 2
					   = \frac{25983}{68000} n^{1/3} - 2 \quad \text{and} \\
		|\mathcal W_2| &= (\frac{1}{2}+\delta+\kappa) n^{1/3} = \frac{42017}{68000} n^{1/3}.
	\end{align*}

	\noindent
	Plugging in the sizes of $\mathcal W_1, \mathcal W_2$ we obtain
	\begin{align*}	
		s &\geq |\mathcal W_1| + (\frac{6483}{17000} - 3n^{-1/3}) |\mathcal W_2| \frac{1000}{999} \\
	  &\geq \frac{25983}{68000} n^{1/3} + (\frac{6483}{17000} - 3n^{-1/3}) \frac{42017}{67932} n^{1/3} - 2 \\
	  &\geq \frac{475777}{769896} n^{1/3} - 4.
	\end{align*}
	Since~$v$ controls a $(1-\kappa) = \frac{3999}{4000}$ fraction of each tile counted by~$s$ and each tile has a size of~$n^{2/3}$,
	we finally have the following lower bound on the second out-neighbourhood of~$v$:
	\begin{align*}	
	  |N^{++}[v]| \geq \frac{3999}{4000} s n^{2/3} 
	  &\geq  \frac{3999}{4000} \cdot \frac{475777}{769896} n - 4n^{2/3} \\
	  &= 0.617\textcolor{amaranth}{82}\ldots n - 4n^{2/3}.
	\end{align*}
	This value lies, for large enough~$n$, above our target value of 
	$(\frac{1}{2} + \delta) n = 0.617\textcolor{amaranth}{64}\ldots n$.	       	  		
\end{proof}


\section{Conclusion}\label{sec:conclusion}

We have shown how the usage of a \emph{template-graph} helped us devise a
seeker strategy that reveals a $(\frac{1}{2}+\frac{2}{17})$-king in a
tournament using $\tilde O(n^{4/3})$ queries, shedding light on a long-standing open
problem. Our approach begins with a
non-adaptive querying strategy based on what we called a \emph{template
graph}, which then helps to guide the seeker to identify a small set of queries which necessarily lead to the discovery of a $(\frac{1}{2}+\frac{2}{17})$-king.

Naturally, we ask whether it is possible to find an improved strategy
which reveals a $(\frac{1}{2}+\delta)$-king with $\delta$ substantially
larger than~$\frac{2}{17}$ using a similar amount of queries. 

\bibliographystyle{plainurl}
\bibliography{main}
	
\end{document}